\documentclass{amsart}
\usepackage{amsfonts}
\usepackage{amsmath}
\usepackage{amssymb}
\usepackage{amsthm}

\newtheorem{theorem}{Theorem}

\newtheorem{lemma}{Lemma}
\newtheorem{rem}{Remark}
\newtheorem{exmp}{Example}
\newtheorem{cor}{Corollary}

\begin{document}

\title{On embeddings of Grassmann graphs in polar Grassmann graphs}
\author{Mark Pankov}
\address{Department of Mathematics and Computer Science, 
University of Warmia and Mazury,
S{\l}oneczna 54, Olsztyn, Poland}
\email{pankov@matman.uwm.edu.pl}

\begin{abstract}
We establish that every embedding of a Grassmann graph 
in a polar Grassmann graph can be reduced to an embedding in a Grassmann graph
or to an embedding in the collinearity graph of a polar space.
Also, we consider $3$-embeddings, 
i.e. embeddings preserving all distances not greater than $3$,
of dual polar graphs whose diameter is not less than $3$
in polar Grassmann graphs formed by non-maximal singular subspaces.
Using the same arguments we show that
every such an embedding can be reduced to an embedding in a Grassmann graph. 
\end{abstract}

\maketitle

\section{Introduction}
In this paper Grassmann graphs and polar Grassmann graphs are considered. 
Almost all polar spaces
can be obtained from sesquilinear, quadratic and pseudo-quadratic forms \cite{Tits} and 
the corresponding polar Grassmann graphs consist of totally isotropic or 
totally singular subspaces of these forms.
Grassmann graphs and polar Grassmann graphs are interesting for many reasons, 
see \cite{BCN-book,BC-book,D-book,P-book1,Pasini-book}.
For example, Grassmannians of vector spaces and polar Grassmannians
can be considered as point-line geometries
whose collinearity graphs are Grassmann and polar Grassmann graphs.
Such geometries are called Grassmann spaces and polar Grassmann spaces,
respectively. 
They are closely connected to buildings of classical types \cite{Tits}.

Grassmann graphs are naturally contained in polar Grassmann graphs
formed by non-maximal singular subspaces.
Consider, for example, the set of all singular subspaces of the same non-maximal dimension
contained in a certain maximal singular subspace.
We describe all possible embeddings of Grassmann graphs in polar Grassmann graphs
(Theorem \ref{theorem-main}).
Some of them are reducible to embeddings in Grassmann graphs.
The existence of other embeddings is related to the fact that
the Grassmann space formed by $2$-dimensional subspaces of a $4$-dimensional vector 
space is a polar space.
Our result is in spirit of the description of 
subspaces in polar Grassmann spaces isomorphic to Grassmann spaces \cite{BC,Coop}, 
but we use different arguments which can be applied to embeddings of other graphs,
for example, to dual polar graphs.

We take a dual polar graph whose diameter is not less than $3$
(the rank of the associated polar space is also not less than $3$)
and consider $3$-embeddings 
(embeddings preserving all distances not greater than $3$)
of this graph in a polar Grassmann graph formed by non-maximal singular subspaces.
Using the same arguments we show  that every such an embedding
can be reduced to an embedding in a Grassmann graph. 
Note that almost all dual polar graphs are naturally contained in Grassmann graphs. 
We do not describe all possible embeddings of dual polar graphs in Grassmann graphs here,
perhaps, it will be a topic of other paper.

It was noted above that some embeddings of Grassmann graphs 
in polar Grassmann graphs are reducible to embeddings in Grassmann graphs.
Embeddings of Grassmann graphs in Grassmann graphs 
are investigated in \cite{P-book2}. 
All such isometric embeddings are known and 
can be obtained from semilinear embeddings of special type.
Non-isometric embeddings also exist.  

All isometric embeddings of dual polar graphs in dual polar graphs
are described in \cite{Pankov1}. 
Similar results for polar Grassmann graphs formed 
by non-maximal singular subspaces can be found in \cite{KP}.

\section{Basic objects}

\subsection{Graphs}
We define a {\it graph} as a pair $\Gamma=({\mathcal V}, \sim)$,
where ${\mathcal V}$  is a non-empty (not necessarily finite) set 
whose elements are called {\it vertices}
and $\sim$ is a symmetric relation on ${\mathcal V}$ called {\it adjacency}.
We assume that our graph does not contain loops, i.e.
$v\not\sim v$ for every vertex $v\in {\mathcal V}$.
A {\it clique} is a subset of ${\mathcal V}$, 
where any two distinct elements are adjacent vertices.
A {\it path} is a sequence of vertices $v_{1},\dots,v_{k}$
such that $v_{i}$ and $v_{i+1}$ are adjacent for every $i\in \{1,\dots,k-1\}$.

We also assume that the graph $\Gamma$ is connected, i.e.
any two distinct vertices can be connected by a path.
Following \cite[Section 15.1]{DD} we define 
the {\it distance} $d(v,w)$ between vertices $v,w\in {\mathcal V}$  
as the smallest number $d$ such that there is a path 
$$v=v_{0},v_{1},\dots,v_{d}=w.$$
Then for every path $v_{1},\dots,v_{i}$ we have $i\ge d(v_{1},v_{i})+1$
and this path is a {\it geodesic} if $i=d(v_{1},v_{i})+1$. 
The {\it diameter} of $\Gamma$ is the greatest distance between two vertices.

An {\it embedding} of a graph $\Gamma=({\mathcal V}, \sim)$ 
in a graph $\Gamma'=({\mathcal V}', \sim)$ 
is an injection of ${\mathcal V}$ to ${\mathcal V}'$
transferring adjacent vertices of $\Gamma$ to adjacent vertices of $\Gamma'$
and non-adjacent vertices of $\Gamma$ to non-adjacent vertices of $\Gamma'$.
Every embedding preserves distance $1$ and $2$
and we say that it is an $m$-{\it embedding} if 
every distance not greater than $m\ge 3$ is preserved.
An embedding is said to be {\it isometric} if it preserves
all distances between vertices.
If the graph diameter is not greater than $2$
then every embedding of this graph is isometric.

\subsection{Partial linear spaces}
A {\it partial linear space} is a pair $\Pi=(P,{\mathcal L})$, 
where $P$ is a non-empty set whose elements are called {\it points}
and ${\mathcal L}$ is a family of proper subsets of $P$ called {\it lines}.
Every line contains at least two points, every point belongs to a certain  line
and for any two distinct points there is at most one line containing them.
We say that two  points are {\it collinear} if they are joined by a line.
A {\it subspace} of $\Pi$ is a subset $S\subset P$ such that for 
any two collinear points of $S$ the line joining them is contained in $S$.
A subspace is said to be {\it singular} if any two distinct points of this subspace are collinear.
The {\it collinearity graph} of $\Pi$ is the graph whose vertex set is $P$
and two distinct points are adjacent vertices of the graph if these points are collinear.

Two partial linear spaces $\Pi=(P,{\mathcal L})$ and $\Pi'=(P',{\mathcal L}')$
are {\it isomorphic} if there is a bijection $f:P\to P'$ such that
$f({\mathcal L})={\mathcal L}'$. 

\subsection{Grassmann graphs}
Let $V$ be an $m$-dimensional vector space over a division ring.
For every integer $i\in \{1,\dots,m-1\}$ we denote by ${\mathcal G}_{i}(V)$
the Grassmannian consisting of $i$-dimensional subspaces of $V$.
The {\it Grassmann graph} $\Gamma_{i}(V)$ 
is the graph whose vertex set is ${\mathcal G}_{i}(V)$ and 
two $i$-dimensional subspaces are adjacent vertices of the graph
if their intersection is $(i-1)$-dimensional.
In the case when $i=1,m-1$, any two distinct vertices of $\Gamma_{i}(V)$ are adjacent.
For this reason we will always suppose that $1<i<m-1$.
The annihilator mapping induces an isomorphism between 
$\Gamma_{i}(V)$ and $\Gamma_{m-i}(V^{*})$,
where $V^{*}$ is the dual vector space. 

The graph $\Gamma_{i}(V)$ is connected.
If $X,Y\in {\mathcal G}_{i}(V)$ then the distance $d(X,Y)$ 
in $\Gamma_{i}(V)$ is equal to 
$$i-\dim(X\cap Y)=\dim(X+Y)-i.$$
In particular, the diameter of $\Gamma_{i}(V)$ is equal to
the minimum of $i$ and $m-i$.

Let $S$ and $U$ be incident subspaces of $V$
such that $\dim S<i<\dim U$.
Denote by $[S,U]_{i}$ the set of all $X\in {\mathcal G}_{i}(V)$
satisfying $S\subset X\subset U$.
If $S=0$ or $U=V$ then we will write $\langle U]_{i}$ or $[S\rangle_{i}$, 
respectively.
In the case when $\dim S=i-1$ and $\dim U=i+1$, 
the set $[S,U]_{i}$ is said to be a {\it line}.

The Grassmann graph $\Gamma_{i}(V)$ has precisely the following two types 
of maximal cliques:
\begin{enumerate}
\item[$\bullet$] the {\it stars} $[S\rangle_{i}$, $S\in {\mathcal G}_{i-1}(V)$,
\item[$\bullet$] the {\it tops} $\langle U]_{i}$, $U\in {\mathcal G}_{i+1}(V)$.
\end{enumerate}
The intersection of two distinct maximal cliques of $\Gamma_{i}(V)$
is empty or a one-element set or a line. 
The third possibility is realized only in the case 
when the maximal cliques are of different types
(one of them is a star and the other is a top) and 
the associated $(i-1)$-dimensional and $(i+1)$-dimensional subspaces are incident.

The {\it Grassmann space} ${\mathfrak G}_{i}(V)$
is the partial linear space whose point set is ${\mathcal G}_{i}(V)$
and whose lines are defined above.
The corresponding collinearity graph is the Grassmann graph $\Gamma_{i}(V)$.
If $S$ and $U$ are incident subspaces of $V$
satisfying $\dim S < i-1$ and $\dim U >i+1$ then 
$[S,U]_{i}$ is a subspace of ${\mathfrak G}_{i}(V)$ isomorphic to ${\mathfrak G}_{j}(U/S)$,
where $j=i-\dim S$. 
Subspaces of such type  are called {\it parabolic}.
By \cite{CKS}, every subspace of ${\mathfrak G}_{i}(V)$ 
isomorphic to a certain Grassmann space is parabolic.

\subsection{Polar Grassmann graphs}
Following \cite{BC-book,P-book1,Shult-book,Ueberberg}
we define a {\it polar space} as a partial linear space satisfying the following axioms:
\begin{enumerate}
\item[(P1)] every line contains at least three points,
\item[(P2)] there is no point collinear to all points,
\item[(P3)] for every point and every line
the point is collinear to one or all points of the line,
\item[(P4)] any chain of mutually distinct incident singular subspaces is finite.
\end{enumerate}
If a polar space has a singular subspace containing more than one line
then all maximal singular subspaces are projective spaces
of the same dimension $n\ge 2$ and the number $n+1$ is called the {\it rank} of 
this polar space.
A polar space is of rank $2$ if all maximal singular subspaces are lines;
such polar spaces are called {\it generalized quadrangles}.
All polar spaces of rank $\ge 3$ are known \cite{Tits}.

\begin{exmp}\label{exmp-1}{\rm
We take any non-degenerate sesquilinear reflexive form (alternating, symmetric or hermitian).
If this form is trace-valued and has totally isotropic subspaces of dimension greater than $1$
then it defines a polar space.
The points of this polar space are $1$-dimensional isotropic subspaces and
the lines are defined by $2$-dimensional totally isotropic subspaces.
There is a natural one-to-one correspondence 
between totally isotropic subspaces of the form and singular subspaces of the polar space.
The rank of the polar space is equal to the dimension of maximal isotropic subspaces.
Similarly, some polar spaces can be obtained from quadratic and pseudo-quadratic forms,
but we get new examples of polar spaces only in the case when
such a form is defined on a vector space over a division ring of characteristic $2$.
}\end{exmp}

By \cite{Tits}, every polar space of rank $\ge 4$ can be obtained from 
a sesquilinear or quadratic or pseudo-quadratic form.
See also \cite{Pasini} for a new approach related to
the notion of a generalized pseudo-quadratic form.

\begin{exmp}\label{exmp-2}{\rm
If $m=4$ then the Grassmann space ${\mathfrak G}_{2}(V)$ is a polar space of rank $3$.
This polar space corresponds to the Klein quadric defined on $\wedge^{2}(V)$
(as it was described in the previous example) if $V$ is a vector space over a field.
}\end{exmp}

Let $\Pi=(P, {\mathcal L})$ be a polar space of rank $n$.
For points $p,q\in P$ we write $p\perp q$ if these points are collinear.
Similarly, if $X$ and $Y$ are subsets of $P$ then 
$X\perp Y$ means that every point of $X$ is collinear to every point of $Y$.
The subspace of $\Pi$ spanned by a subset $X\subset P$, i.e.
the minimal subspace containing $X$, is denoted by $\langle X\rangle$.
In the case when $X\perp X$, this subspace is singular.

For every $k\in \{0,1,\dots,n-1\}$ we denote by ${\mathcal G}_{k}(\Pi)$
the polar Grassmannian consisting of $k$-dimensional singular subspaces of $\Pi$.
Then ${\mathcal G}_{0}(\Pi)=P$
and ${\mathcal G}_{n-1}(\Pi)$ is formed by maximal singular subspaces.
The {\it polar Grassmann graph $\Gamma_{k}(\Pi)$}
is the graph whose vertex set is ${\mathcal G}_{k}(\Pi)$.
In the case when  $k\le n-2$, two distinct elements of ${\mathcal G}_{k}(\Pi)$
are adjacent vertices of $\Gamma_{k}(\Pi)$
if there is a $(k+1)$-dimensional singular subspace containing them.
Two distinct maximal singular subspaces are adjacent vertices of $\Gamma_{n-1}(\Pi)$
if their intersection is $(n-2)$-dimensional. 
The graph $\Gamma_{n-1}(\Pi)$ is known as the {\it dual polar graph} of $\Pi$.
The graph $\Gamma_{0}(\Pi)$ is the collinearity graph of $\Pi$.

The graph $\Gamma_{k}(\Pi)$ is connected for every $k$.
If $X,Y\in {\mathcal G}_{n-1}(\Pi)$ then the distance $d(X,Y)$ 
in $\Gamma_{n-1}(\Pi)$ is equal to
$$n-1-\dim(X\cap Y)$$
and the diameter of $\Gamma_{n-1}(\Pi)$ is equal to $n$
(the dimension of the empty set is $-1$).
In the case when $k\le n-2$, the diameter of $\Gamma_{k}(\Pi)$ is equal to $k+2$
and there is the following description of the distance. 

\begin{lemma}[\cite{KP}]\label{lemma-dist}
If $X,Y\in {\mathcal G}_{k}(\Pi)$, $k\le n-2$ and 
the distance between $X$ and $Y$ in $\Gamma_{k}(\Pi)$
is equal to $d$ then one of the following possibilities is realized:
\begin{enumerate}
\item[$\bullet$] $\dim(X\cap Y)=k-d$ and
there exist points $p\in X\setminus Y$ and $q\in Y\setminus X$ such that
$p\perp Y$ and $q\perp X$,
\item[$\bullet$] $d>1$, $\dim(X\cap Y)=k-d+1$,
for every point $p\in X\setminus Y$ there is a point of $Y$ non-collinear to $p$
and for every point $q\in Y\setminus X$ there is a point of $X$ non-collinear to $q$.
\end{enumerate}
If $d=k+2$ then only the second possibility is realized.
\end{lemma}

For every singular subspace $S$ of dimension less than $k$
we denote by $[S\rangle_{k}$ the set of all elements of ${\mathcal G}_{k}(\Pi)$
containing $S$. 
If $S\in {\mathcal G}_{n-2}(\Pi)$ then the set $[S\rangle_{n-1}$
is called a {\it line} of ${\mathcal G}_{n-1}(\Pi)$. 
Every maximal clique of $\Gamma_{n-1}(\Pi)$ is a line. 

Suppose that $k\le n-2$.
Let $S$ and $U$ be incident singular subspaces such that 
$\dim S <k<\dim U$.
Denote by $[S,U]_{k}$ the set of all $X\in {\mathcal G}_{k}(\Pi)$
satisfying $S\subset X\subset U$.
If $S=\emptyset$ then we write $\langle U]_{k}$.
In the case when $\dim S=k-1$ and $\dim U=k+1$,
the set $[S,U]_{k}$ is called a {\it line} of ${\mathcal G}_{k}(\Pi)$
(this is a line of $\Pi$ if $k=0$).

Every maximal clique of $\Gamma_{0}(\Pi)$ is a maximal singular subspace of $\Pi$. 
The polar Grassmann graph $\Gamma_{k}(\Pi)$, $1\le k \le n-3$
has precisely the following two types of maximal cliques:
\begin{enumerate}
\item[$\bullet$] 
the {\it stars} $[S,U]_{k}$, $S\in {\mathcal G}_{k-1}(\Pi)$
and $U\in {\mathcal G}_{n-1}(\Pi)$,
\item[$\bullet$] the {\it tops} $\langle U]_{k}$, $U\in {\mathcal G}_{k+1}(\Pi)$.
\end{enumerate}
If $k=n-2$ then all stars are lines and every maximal clique of $\Gamma_{n-2}(\Pi)$
is a top. 

The {\it polar Grassmann space} ${\mathfrak G}_{k}(\Pi)$ is
the partial linear space whose point set is ${\mathcal G}_{k}(\Pi)$
and whose lines are defined above. 
The corresponding collinearity graph is the polar Grassmann graph $\Gamma_{k}(\Pi)$.
For every $(k-1)$-dimensional singular subspace $S$ of $\Pi$
the subspace $[S\rangle_{k}$ is a polar space of rank $n-k$.
We denote this polar space by $\Pi_{S}$. 
If $S$ and $U$ are incident singular subspaces of $\Pi$
such that $\dim S<k-1$ and $\dim U > k+1$
then $[S,U]_{k}$ is a subspace of ${\mathfrak G}_{k}(\Pi)$
 isomorphic to a Grassmann space.
Subspaces of such type are called {\it parabolic} \cite{BC,Coop}.

\section{Main result}

Our main result says that there are precisely two types 
of embeddings of Grassmann graphs in polar Grassmann graphs.

\begin{theorem}\label{theorem-main}
If $f$ is an embedding of $\Gamma_{i}(V)$ in $\Gamma_{k}(\Pi)$
then $k\le n-3$ and one of the following possibilities is realized: 
\begin{enumerate}
\item[{\rm (A)}] $k\ge 1$ and there is a singular subspace $U$ of $\Pi$ such that
the image of $f$ is contained in $\langle U]_{k}$,
\item[{\rm (B)}] there is a $(k-1)$-dimensional singular subspace $S$ of $\Pi$
such that the image of $f$ is contained in $[S\rangle_{k}$.
\end{enumerate}
\end{theorem}

Suppose that $f$ is an embedding of  type (A), i.e.
$k\ge 1$ and the image of $f$ is contained in $\langle U]_{k}$,
where $U$ is a singular subspace of $\Pi$.
By Theorem \ref{theorem-main}, $n\ge 4$ which implies that $\Pi$ is defined by a form
(sesquilinear, quadratic or pseudo-quadric). 
Then $U$ is the projective space associated to a certain vector space $W$.
The restriction of the graph $\Gamma_{k}(\Pi)$ to $\langle U]_{k}$
can be naturally identified with the Grassmann graph $\Gamma_{k+1}(W)$
and $f$ is an embedding of $\Gamma_{i}(V)$ in $\Gamma_{k+1}(W)$.

\begin{rem}{\rm
Embeddings of Grassmann graphs in Grassmann graphs
are investigated in \cite[Chapter 3]{P-book2}. 
All isometric embeddings are known and can be obtained from
semilinear embeddings of special type \cite[Theorem 3.1]{P-book2}.
These semilinear embeddings are not necessarily strong
(a strong semilinear embedding transfers any collection of linearly independent vectors 
to linearly independent vectors). 
Also, some non-strong semilinear embeddings
define non-isometric embeddings of Grassmann graphs.
The classical Chow theorem concerning automorphisms of Grassmann graphs 
\cite{Chow} easily follows from the mentioned above description 
of isometric embeddings. 
Interesting results related to Chow's theorem
can be found in \cite{HVM, Havlicek, HLW, Huang, Kreuzer, Lim}.
}\end{rem}

\begin{exmp}{\rm
If $m=4$ then the Grassmann space ${\mathfrak G}_{2}(V)$
is a polar space and every collinearity preserving injection of this polar space 
to the polar space $\Pi_{S}$, $S\in {\mathcal G}_{k-1}(\Pi)$
is an embedding of $\Gamma_{2}(V)$ in $\Gamma_{k}(\Pi)$.
By \cite[Examples 1.18 and Proposition 3.4]{P-book2}, 
there is a semilinear embedding $l:W\to W'$
such that $\dim W>\dim W'=4$ and 
$l$ induces an embedding of $\Gamma_{2}(W)$ in $\Gamma_{2}(W')$.
Therefore, embeddings of $\Gamma_{2}(V)$ in $\Gamma_{k}(\Pi)$ of type (B)
exist in the case when $m>4$.
}\end{exmp}

The distance between two elements of $[S\rangle_{k}$, $S\in {\mathcal G}_{k-1}(\Pi)$ 
in the graph $\Gamma_{k}(\Pi)$ is not greater than $2$.
Thus an isometric embedding of $\Gamma_{i}(V)$ in $\Gamma_{k}(\Pi)$
can be of type (B) only in the case when the diameter of 
the Garassmann graph $\Gamma_{i}(V)$ is equal to $2$,
i.e. $i=2,m-2$. 
So, if $3\le i\le m-3$ then every isometric embedding of $\Gamma_{i}(V)$ in $\Gamma_{k}(\Pi)$
is of type (A) and, by \cite[Theorem 3.1]{P-book2}, 
it is induced by a semilinear embedding.

Recall that every subspace of ${\mathfrak G}_{i}(V)$ 
isomorphic to a Grassmann space is parabolic \cite{CKS}.
Using this fact and Theorem \ref{theorem-main}, we prove the following.

\begin{cor}
If ${\mathcal S}$ is a subspace of ${\mathfrak G}_{k}(\Pi)$ isomorphic to ${\mathfrak G}_{i}(V)$
then one of the following possibilities is realized: 
\begin{enumerate}
\item[$\bullet$]  ${\mathcal S}$ is a parabolic subspace,
\item[$\bullet$]  ${\mathcal S}$ is contained in $[S\rangle_{k}$, $S\in {\mathcal G}_{k-1}(\Pi)$
and $m=4$, $i=2$.
\end{enumerate}
\end{cor}

\begin{proof}
Every isomorphism of ${\mathfrak G}_{i}(V)$ to ${\mathcal S}$
is an embedding of $\Gamma_{i}(V)$ in $\Gamma_{k}(\Pi)$.
If this embedding is of type (A) then 
${\mathcal S}$ is contained in a parabolic subspace of ${\mathfrak G}_{k}(\Pi)$.
Every parabolic subspace of ${\mathfrak G}_{k}(\Pi)$ is naturally isomorphic 
to a Grassmann space
and the above mentioned result from \cite{CKS}
guarantees that ${\mathcal S}$ is parabolic.
If the embedding is of type (B) then
${\mathcal S}$ is contained in $[S\rangle_{k}$, $S\in {\mathcal G}_{k-1}(\Pi)$.
The subspace $[S\rangle_{k}$ is a polar space and the same holds for ${\mathcal S}$ 
which is possible only in the case when $m=4$ and $i=2$.
\end{proof}

\begin{rem}{\rm
Similar results for the polar spaces defined by sesquilinear forms 
were obtained in \cite{BC,Coop}.
Also, \cite{BC} describes all cases when the second possibility is realized.
}\end{rem}

\begin{rem}{\rm 
Note that ${\mathcal S}$ can be a proper subspace of 
$[S\rangle_{k}$, $S\in {\mathcal G}_{k-1}(\Pi)$
even if $[S\rangle_{k}$ is a polar space of rank $3$.
Consider the Klein quadric defined on a $6$-dimensional vector space over a field of characteristic $2$.
The associated sesquilinear form is alternating.
This form is also non-degenerate and trace-valued.
The corresponding polar space contains the polar space defined by the Klein quadric.  
}\end{rem}

\begin{rem}{\rm
Now we consider the case when ${\mathcal S}$ is contained in the polar space
$[S\rangle_{k}$, $S\in {\mathcal G}_{k-1}(\Pi)$ whose rank is greater than $3$. 
We take any frame ${\mathcal F}$ of the polar space ${\mathcal S}$
and extend it to a frame ${\mathcal F}'$ of the polar space $[S\rangle_{k}$
(every frame of a rank $n$ polar space is formed by 
$2n$ distinct points $p_{1},\dots,p_{2n}$ such that 
for each $p_{i}$ there is unique $p_{j}$ satisfying $p_{i}\not\perp p_{j}$).
Since ${\mathcal S}$ is isomorphic to a Grassmann space,
${\mathcal F}$ corresponds to an apartment in the associated Grassmannian
and ${\mathcal S}$ is spanned by ${\mathcal F}$
(see \cite[Subsections 3.3.1 and 4.3.3]{P-book1} for the details).
This implies that ${\mathcal S}$ is contained in the rank $3$ polar space
formed by all points of $[S\rangle_{k}$ collinear 
to every point of ${\mathcal F}'\setminus {\mathcal F}$.
By the previous remark, 
${\mathcal S}$ can be a proper subspace of this polar space.
}\end{rem}

\section{Proof of Theorem \ref{theorem-main}} 

\subsection{Preliminary observations}
Let $f$ be an embedding of $\Gamma_{i}(V)$ in $\Gamma_{k}(\Pi)$.
Then $f$ transfers maximal cliques of $\Gamma_{i}(V)$ to not necessarily maximal cliques
of $\Gamma_{k}(\Pi)$, i.e. subsets of maximal cliques.
If ${\mathcal X}$ and ${\mathcal Y}$ are distinct maximal cliques of $\Gamma_{i}(V)$
then there exist $X\in {\mathcal X}$ and $Y\in {\mathcal Y}$ 
which are not adjacent in $\Gamma_{i}(V)$.
Since $f(X)$ and $f(Y)$ are non-adjacent vertices of $\Gamma_{k}(\Pi)$,
there is no maximal clique of $\Gamma_{k}(\Pi)$ containing 
$f({\mathcal X})$ and $f({\mathcal Y})$.
So, distinct maximal cliques go to subsets of distinct maximal cliques.

If ${\mathcal X}$ is a star of ${\mathcal G}_{i}(V)$
and ${\mathcal Y}$ is a top of ${\mathcal G}_{i}(V)$
such that the associated $(i-1)$-dimensional and $(i+1)$-dimensional 
subspaces of $V$ are incident then ${\mathcal X}\cap {\mathcal Y}$ is a line of 
${\mathcal G}_{i}(V)$.
This means that $f({\mathcal X})\cap f({\mathcal Y})$ contains more than one 
element and the same holds for the intersection of maximal cliques of $\Gamma_{k}(\Pi)$
containing $f({\mathcal X})$ and $f({\mathcal Y})$.
Therefore, $k\le n-3$. 
Indeed, if $k$ is equal to $n-2$ or $n-1$ then all maximal cliques of 
$\Gamma_{k}(\Pi)$ are of the same type
and the intersection of two distinct maximal cliques contains at most one element. 

\begin{lemma}\label{lemma-line}
If ${\mathcal X}$ is a maximal clique of $\Gamma_{i}(V)$
then $f({\mathcal X})$ is not contained in a line of ${\mathcal G}_{k}(\Pi)$.
\end{lemma}

\begin{proof}
We take three distinct elements $A,B,C\in {\mathcal X}$
which are not on a common line of ${\mathcal G}_{i}(V)$. 
There is the unique maximal clique ${\mathcal Y}$
intersecting ${\mathcal X}$ precisely in the line containing $A$ and $B$.
Every  $D\in {\mathcal Y}\setminus {\mathcal X}$ 
is a vertex of $\Gamma_{i}(V)$ adjacent to $A,B$  and non-adjacent to $C$.
Then $f(D)$ is a vertex of $\Gamma_{k}(\Pi)$
adjacent to $f(A),f(B)$ and non-adjacent to $f(C)$.
An easy verification shows that $f(D)$ is adjacent to all vertices of $\Gamma_{k}(\Pi)$
belonging to the line joining $f(A)$ and $f(B)$. 
Hence $f(C)$ is not on this line.
\end{proof}

Lemma \ref{lemma-line} shows that 
the image of a maximal clique cannot be contained in two maximal cliques of different types. 
Indeed, if the intersection of a star and a top of ${\mathcal G}_{k}(\Pi)$ 
contains  more than one element then this intersection is a line.

\begin{lemma}\label{lemma-cliques}
If there exists a maximal singular subspace $U$ of $\Pi$ such that
the image of $f$ is contained in $\langle U]_{k}$
then one of the following possibilities is realized:
\begin{enumerate}
\item[$\bullet$] 
$f$ transfers stars of ${\mathcal G}_{i}(V)$ to subsets in stars of ${\mathcal G}_{k}(\Pi)$
and tops of ${\mathcal G}_{i}(V)$ to  subsets in tops of ${\mathcal G}_{k}(\Pi)$,
\item[$\bullet$] 
$f$ transfers stars of ${\mathcal G}_{i}(V)$ to subsets in tops of ${\mathcal G}_{k}(\Pi)$
and tops of ${\mathcal G}_{i}(V)$ to subsets in stars of ${\mathcal G}_{k}(\Pi)$.
\end{enumerate}
\end{lemma}

\begin{proof}
The restriction of the graph $\Gamma_{k}(\Pi)$ to $\langle U]_{k}$,
is naturally identified with a Grassmann graph.
Its maximal cliques are stars and tops of ${\mathcal G}_{k}(\Pi)$
contained in $\langle U]_{k}$. 
The required statement follows from \cite[Proposition 3.5]{P-book2}.
\end{proof}

Let $X$ and $Y$ be elements of ${\mathcal G}_{k}(\Pi)$
such that the distance between $X$ and $Y$ in $\Gamma_{k}(\Pi)$ is equal to $2$.
Since $k\le n-3$, Lemma \ref{lemma-dist} shows that 
one of the following possibilities is realized:
\begin{enumerate}
\item[(1)] $X\cap Y$ is $(k-2)$-dimensional and $X\perp Y$;
\item[(2)] $X\cap Y$ is $(k-2)$-dimensional and $X\not\perp Y$,
but there exist points $p\in X\setminus Y$ and $q\in Y\setminus X$ 
such that $p\perp Y$ and $q\perp X$;
\item[(3)] $X\cap Y$ is $(k-1)$-dimensional and $X\not\perp Y$.
\end{enumerate}

\begin{lemma}\label{lemma-case2}
In the case {\rm (2)}, there is a unique vertex of $\Gamma_{k}(\Pi)$
adjacent to both $X$ and $Y$, i.e. there is only one geodesic joining $X$ and $Y$.
\end{lemma}

\begin{proof}
Let $X'$ be the singular subspace formed by all points of $X$ collinear to all points of $Y$.
Similarly, we denote by $Y'$ the singular subspace consisting of all points of $Y$
collinear to all points of $X$.
In the case (2), these subspaces both are $(k-1)$-dimensional and contain $X\cap Y$.
It is clear that $X'\perp Y'$.
The $k$-dimensional singular subspace spanned by $X'$ and $Y'$
is the unique vertex of $\Gamma_{k}(\Pi)$ adjacent to both $X$ and $Y$.
\end{proof}

\begin{lemma}\label{lemma-cases1-3}
If $A$ and $B$ are vertices of $\Gamma_{i}(V)$  satisfying $d(A,B)=2$
then for $X=f(A)$ and $Y=f(B)$ we have $(1)$ or $(3)$.
\end{lemma}

\begin{proof}
The embedding $f$ transfers every geodesic of $\Gamma_{i}(V)$
consisting of $3$ vertices to a geodesic of $\Gamma_{k}(\Pi)$.
Since $\Gamma_{i}(V)$ contains more than one vertex adjacent to both 
$A$ and $B$, Lemma \ref{lemma-case2} shows that 
the case (2) is impossible. 
\end{proof}

\subsection{Embeddings of type (A)}
Suppose that $f(A)\perp f(B)$ for any two vertices $A$ and $B$ of $\Gamma_{i}(V)$
satisfying $d(A,B)=2$.
In this case, we assert that $f$ is an embedding of type (A), 
i.e. there exists a singular subspace $U$ of $\Pi$ such that the image of $f$
is contained in $\langle U]_{k}$. 
This is equivalent to the fact that $f(A)\perp f(B)$ for any $A,B\in {\mathcal G}_{i}(V)$. 
We prove the latter statement by induction on $d=d(A,B)$.

The statement is obvious for $d=1$
and it follows from our assumption if $d=2$.
Let $d\ge 3$.
Then there exist $D,C_{1},C_{2}\in {\mathcal G}_{i}(V)$ such that 
$$d(A,D)=d-2,\;\;d(D,B)=d(C_{1},C_{2})=2$$
and every $C_{j}$ is adjacent to both $D$ and $B$.
Since $d(A,C_{j})=d-1$ for every $j=1,2$, we have 
$f(A)\perp f(C_{j})$ by the inductive hypothesis. 
Also, we have $f(C_{1})\perp f(C_{2})$ by our assumption.
Let $U$ be the singular subspace of $\Pi$ spanned by 
$$f(A),f(C_{1}), f(C_{2}).$$
Recall that $B$ is adjacent to both $C_{j}$.
This means that $f(B)$ and $f(C_{j})$ are adjacent vertices of $\Gamma_{k}(\Pi)$
and 
$$\dim (f(B)\cap f(C_{j}))=k-1$$
for every $j=1,2$.
We observe that
\begin{equation}\label{eq1}
f(B)\cap f(C_{1})\;\mbox{ and }\;f(B)\cap f(C_{2})
\end{equation}
are distinct $(k-1)$-dimensional subspaces of $f(B)$.
Indeed, if these subspaces are coincident then 
the intersection of $f(C_{1})$ and $f(C_{2})$ is $(k-1)$-dimensional
which is impossible, since $f(C_{1})\perp f(C_{2})$
and the distance between $f(C_{1})$ and $f(C_{2})$
in $\Gamma_{k}(V)$ is equal to $2$.
So, $f(B)$ is spanned by the subspaces \eqref{eq1}. 
These subspaces are contained in $U$ and we get the inclusion $f(B)\subset U$
which implies that $f(A)\perp f(B)$.  

\subsection{Embeddings of type (B)}
Suppose that there exist $A,B\in {\mathcal G}_{i}(V)$ such that 
$$d(A,B)=2\;\mbox{ and }\;f(A)\not\perp f(B).$$
By Lemma \ref{lemma-cases1-3}, the subspace
$$S:=f(A)\cap f(B)$$
is $(k-1)$-dimensional.
We show that the image of $f$ is contained in $[S\rangle_{k}$.

First, we establish this statement for $i=2,m-2$. 
Since the graphs $\Gamma_{2}(V)$ and $\Gamma_{m-2}(V^{*})$
are isomorphic, we can restrict ourself to the case $i=2$.

So, let $i=2$. 
Consider the $4$-dimensional subspace $M:=A+B$.
Our first step is to show that $f(\langle M]_{2})$ is contained in $[S\rangle_{k}$.

If $B$ is a base of $M$ then the associated {\it apartment} of ${\mathcal G}_{2}(M)$
is formed by all $2$-dimensional subspaces spanned by two distinct vectors from $B$.
Every apartment of ${\mathcal G}_{2}(M)$ consists of $6$ elements.
Two apartments of ${\mathcal G}_{2}(M)$ are {\it adjacent} if their intersection
contains precisely $4$ elements.

\begin{lemma}\label{lemma-ap}
For any two apartments ${\mathcal A},{\mathcal A}'\subset {\mathcal G}_{2}(M)$
there is a sequence of apartments 
$${\mathcal A}={\mathcal A}_{1},{\mathcal A}_{2},\dots, {\mathcal A}_{j}={\mathcal A}'$$
such that ${\mathcal A}_{t}$ and ${\mathcal A}_{t+1}$ are adjacent
for every $t\in \{1,\dots,j-1\}$.
\end{lemma}

\begin{proof}
This is a partial case of \cite[Proposition 5.1]{P-book2}.
\end{proof}

We take any apartment ${\mathcal A}\subset {\mathcal G}_{2}(M)$
containing both $A$ and $B$ (an easy verification shows that such an apartment exists).
Then every 
$C\in {\mathcal A}\setminus\{A,B\}$ is adjacent to both $A$ and $B$. 
This guarantees that $f(C)$ contains $S$, i.e.
$$f({\mathcal A})\subset [S\rangle_{k}.$$
Since every element of ${\mathcal G}_{2}(M)$ is contained in a certain apartment,
we need to show that the same inclusion holds for every apartment 
${\mathcal A}'\subset {\mathcal G}_{2}(M)$.
By Lemma \ref{lemma-ap}, it is sufficient to consider the case when
${\mathcal A}$ and ${\mathcal A}'$ are adjacent apartments.
Then $|{\mathcal A}\cap {\mathcal A}'|=4$ 
which implies the existence of $A',B'\in {\mathcal A}\cap {\mathcal A}'$ satisfying $d(A',B')=2$.
Since $A',B'\in {\mathcal A}$, 
we have $f(A'),f(B')\in[S\rangle_{k}$.
As above, every 
$C'\in {\mathcal A}'\setminus\{A',B'\}$ is adjacent to both $A',B'$
and we establish that $f(C')$ belongs to $[S\rangle_{k}$.

Therefore, $f(\langle M]_{2})$ is contained in $[S\rangle_{k}$.
If $m=4$ then $M=V$ and we get the claim.
In the case when $m>4$, we need to show that
$$f(\langle N]_{2})\subset [S\rangle_{k}$$
for every $4$-dimensional subspace $N\subset V$. 
Since the graph $\Gamma_{4}(V)$ is connected, 
it is sufficient to prove the latter inclusion only
for the case when $M$ and $N$ are adjacent vertices of $\Gamma_{4}(V)$.

Let $N$ be a vertex of $\Gamma_{4}(V)$ adjacent to $M$.
Then $T:=M\cap N$ is $3$-dimensional and 
$\langle T]_{2}$ is a top in the both Grassmannians 
${\mathcal G}_{2}(M)$ and ${\mathcal G}_{2}(N)$.
The image of this top is contained in $[S\rangle_{i}$.
It follows from Lemma \ref{lemma-line} that 
$f(\langle T]_{2})$ is a subset of a certain star $[S,Q]_{k}$.

Suppose that for all $A',B'\in \langle N]_{2}$ satisfying $d(A',B')=2$
we have $\noindent{f(A')\perp f(B')}$.
By Subsection 4.2, the restriction of $f$ to $\langle N]_{2}$ is an embedding of type (A).
We apply Lemma \ref{lemma-cliques} to this embedding.
Since $f(\langle T]_{2})$ is contained in a star, 
$f$ transfers every star of ${\mathcal G}_{2}(N)$ to a subset in a top.
Let $P$ be a $1$-dimensional subspace of $T$.
Then $[P,M]_{2}$ and $[P,N]_{2}$ are stars in ${\mathcal G}_{2}(M)$
and ${\mathcal G}_{2}(N)$, respectively.
Our embedding sends the first star to a subset in a certain star ${\mathcal S}=[S,K]_{k}$ 
and the second star goes to a subset in a top ${\mathcal T}$.
Every element of ${\mathcal S}\setminus {\mathcal T}$
is non-adjacent to every element of ${\mathcal T}\setminus {\mathcal S}$.
This means that $$f([P,M]_{2})\cup f([P,N]_{2})$$
is contained in ${\mathcal S}$ or ${\mathcal T}$.
Then $f([P,M]_{2})$ or $f([P,N]_{2})$ is a subset of the line  ${\mathcal S}\cap {\mathcal T}$
which is impossible.

Thus there exist $A',B'\in \langle N]_{2}$ such that $f(A')\not \perp f(B')$.
By the arguments given above, 
we have 
$$f(\langle N]_{2})\subset [S'\rangle_{k}$$
for a certain $S'\in {\mathcal G}_{k-1}(\Pi)$.
Then $f(\langle T]_{2})$ is contained in both $[S\rangle_{k}$ and $[S'\rangle_{k}$,
i.e.
the intersection of $[S\rangle_{i}$ and $[S'\rangle_{i}$ contains more than one element.
If $S$ and $S'$ are distinct then this intersection contains at most one element.
Hence $S=S'$.

So, $f$ is an embedding of type (B) if $i=2,m-2$.
In the case when $2<i<m-2$, the subspace $M:=A+B$ is $(i+2)$-dimensional.
Therefore, the restriction of $f$ to $\langle M]_{i}$ is an embedding of type (B)
whose image is contained in $[S\rangle_{k}$.
Let us take any $N\in {\mathcal G}_{i+2}(V)$ adjacent to $M$ in $\Gamma_{i+2}(V)$.
As for the case $i=2$, we show that 
the restriction of $f$ to $\langle N]_{i}$ is an embedding of type (B)
and its image is contained in $[S\rangle_{k}$.
Since the graph $\Gamma_{i+2}(V)$ is connected,
the same holds for every $N\in {\mathcal G}_{i+2}(V)$.

\section{Embeddings of dual polar graphs in polar Grassmann graphs 
formed by non-maximal singular subspaces}

Using some arguments from the previous section,
we prove the following.

\begin{theorem}\label{theorem2}
Let $\Pi'=(P',{\mathcal L}')$ be a polar space of rank $l\ge 3$.
If $f$ is a $3$-embedding of the dual polar graph $\Gamma_{l-1}(\Pi')$
in the polar Grassmann graph $\Gamma_{k}(\Pi)$, $k\le n-2$
then $k\le n-4$ and there is a singular subspace $U$ of $\Pi$
such that the image of $f$ is contained in $\langle U]_{k}$.
\end{theorem}

\begin{rem}{\rm
If a polar space of rank $l$ is related to a certain form 
(sesquilinear, quadratic or pseudo-quadric) defined on a vector space $W$
then the associated dual polar graph is naturally contained in
the Grassmann graph $\Gamma_{l}(W)$.
}\end{rem}

Since $k\le n-2$, for any two vertices $X$ and $Y$ of $\Gamma_{k}(\Pi)$
satisfying $d(X,Y)=2$ one of the following possibilities is realized:
\begin{enumerate}
\item[(1)] $X\cap Y$ is $(k-2)$-dimensional and $X\perp Y$;
\item[(2)] $X\cap Y$ is $(k-2)$-dimensional and $X\not\perp Y$,
but there exist points $p\in X\setminus Y$ and $q\in Y\setminus X$ 
such that $p\perp Y$ and $q\perp X$;
\item[(3)] $X\cap Y$ is $(k-1)$-dimensional and $X\not\perp Y$.
\end{enumerate}
Lemma \ref{lemma-case2} states that in the case (2)
there is the unique vertex of $\Gamma_{k}(\Pi)$ adjacent to both $X$ and $Y$.
In addition to this fact, we will need the following statement concerning 
the case (3).

\begin{lemma}\label{lemma-case3}
In the case {\rm (3)}, every geodesic of $\Gamma_{k}(\Pi)$ 
between $X$ and $Y$ cannot be extended to a geodesic containing more than three vertices.
\end{lemma}

\begin{proof}
Suppose that $Z$ is a vertex of $\Gamma_{k}(\Pi)$ adjacent to $Y$
and such that the distance between $X$ and $Z$ in $\Gamma_{k}(\Pi)$ 
is equal to $3$.
Then $Z\cap Y$ is $(k-1)$-dimensional and
it follows from Lemma \ref{lemma-dist} that $Z\cap Y$ and $X\cap Y$ 
are distinct $(k-1)$-dimensional subspaces
(otherwise, the subspace $X\cap Z$ is $(k-1)$-dimensional 
which contradicts the fact that $d(X,Z)=3$).
Thus the intersection of $Z\cap Y$ and $X\cap Y$ is $(k-2)$-dimensional
which guarantees that $X\cap Z$ is $(k-2)$-dimensional.
Also, there is a point $p\in X\cap Y$ which does not belong to $Z$.
Since $p\in Y$ and $Y,Z$ are adjacent vertices of $\Gamma_{k}(\Pi)$,
we have $p\perp Z$.
So,
$$\dim (X\cap Z)=k-2$$
and $p$ is a point of $X\setminus Z$ collinear to all points of $Z$.
By Lemma \ref{lemma-dist},
the distance between $X$ and $Z$ in $\Gamma_{k}(\Pi)$ 
is equal to $2$ which contradicts our assumption.
\end{proof}

If $l\ge 3$ then the diameter of $\Gamma_{l-1}(\Pi')$ is not less than $3$
and for every $3$-embedding $f$ of $\Gamma_{l-1}(\Pi')$ in $\Gamma_{k}(\Pi)$
there is the following analogue of Lemma \ref{lemma-cases1-3}.

\begin{lemma}\label{lemma-case1}
If $A$ and $B$ are vertices of $\Gamma_{l-1}(\Pi')$  satisfying $d(A,B)=2$
then for $X=f(A)$ and $Y=f(B)$ we have $(1)$.
\end{lemma}

\begin{proof}
There is more than one vertex of $\Gamma_{l-1}(\Pi')$
adjacent to both $A$ and $B$. 
As in the proof of Lemma \ref{lemma-cases1-3},
we establish that (2) is impossible.
The diameter of $\Gamma_{l-1}(\Pi')$ is not less than $3$
and every geodesic between $A$ and $B$ can be extended to a geodesic
containing $4$ vertices. Since $f$ is a $3$-embedding, it transfers 
this geodesic to a geodesic of $\Gamma_{k}(\Pi)$.
Lemma \ref{lemma-case3} shows that (3) fails.
\end{proof}

Let $A$ and $B$ be vertices of $\Gamma_{l-1}(\Pi')$.
It is clear that $f(A)\perp f(B)$ if these vertices are adjacent.
By Lemma \ref{lemma-case1}, the same holds in the case when $d(A,B)=2$. 
If $d(A,B)\ge 3$ then there exist $D,C_{1},C_{2}\in {\mathcal G}_{l-1}(\Pi')$ such that 
$$d(A,D)=d(A,B)-2,\;\;d(D,B)=d(C_{1},C_{2})=2$$
and every $C_{j}$ is adjacent to both $D$ and $B$.
As in Subsection 4.2, we show that $f(A)\perp f(B)$ by induction.
So, there is a singular subspace $U$ of $\Pi$
such that the image of $f$ is contained in $\langle U]_{k}$. 
Since $f$ is a $3$-embedding, 
the diameter of $\langle U]_{k}$ in $\Gamma_{k}(\Pi)$ is not less than $3$.
The latter implies that $k\le n-4$.

\begin{rem}{\rm
The direct analogue of Theorem \ref{theorem2} holds for 
every graph $\Gamma=({\mathcal V},\sim)$ satisfying the following 
conditions:
\begin{enumerate}
\item[$\bullet$] if $v,w\in {\mathcal V}$ and $d(v,w)=2$ then
there is more than one vertex of $\Gamma$ adjacent to both $v,w$
and every geodesic between $v$ and $w$ can be extended to 
a geodesic containing $4$ vertices,
\item[$\bullet$] if $v,w\in {\mathcal V}$ and $d(v,w)\ge 3$ then
there exist $u,c_{1},c_{2}\in {\mathcal V}$
such that 
$$d(v,u)=d(v,w)-2,\;\;d(u,w)=d(c_{1},c_{2})=2$$
and every $c_{j}$ is adjacent to both $u$ and $w$.
\end{enumerate}
}\end{rem}

\end{document}